\documentclass[twoside,bezier,12pt, reqno]{amsart}
\usepackage{amssymb,latexsym,epsfig,mathrsfs,graphpap,graphics,amsmath,amssymb}
\usepackage{amstext}
\usepackage{amsthm}
\usepackage[utf8]{inputenc}
\usepackage[english]{babel}
\usepackage{helvet}
\usepackage{courier}

\newtheorem{theorem}{Theorem}[section]
\newtheorem{lemma}{Lemma}[section]

\theoremstyle{Definition}
\newtheorem{definition}{Definition}[section]
\newtheorem{example}{Example}[section]

\theoremstyle{remark}
\newtheorem{remark}[theorem]{Remark}
\numberwithin{equation}{section}

\begin{document}

\begin{flushleft}
 {\bf\Large {Towards  Quaternion Quadratic-Phase Fourier Transform  }}

\parindent=0mm \vspace{.2in}

{\bf{Aamir H. Dar$^{1},$ and M. Younus Bhat $^{2}$ }}
\end{flushleft}

{{\it $^{1}$ Department of  Mathematical Sciences,  Islamic University of Science and Technology Awantipora, Pulwama, Jammu and Kashmir 192122, India.E-mail: $\text{ahdkul740@gmail.com}$}}

{{\it $^{2}$ Department of  Mathematical Sciences,  Islamic University of Science and Technology Awantipora, Pulwama, Jammu and Kashmir 192122, India.E-mail: $\text{gyounusg@gmail.com}$}}

\begin{quotation}
\noindent
{\footnotesize {\sc Abstract.} The quadratic-phase Fourier transform (QPFT) is a neoteric addition to the
class of Fourier transforms and embodies a variety of signal processing tools
including the Fourier, fractional Fourier, linear canonical, and special affine
Fourier transform. In this paper, we generalize the quadratic-phase Fourier transform to quaternion-valued signals, known as the
quaternion QPFT (Q-QPFT). We
initiate our investigation by studying the QPFT of 2D quaternionic signals, then we introduce the Q-QPFT of 2D quaternionic signals.
Using the fundamental relationship between the Q-QPFT and quaternion Fourier transform (QFT), we derive the inverse transform and Parseval and Plancherel formulas associated with the Q-QPFT. Some other
properties including linearity, shift and modulation  of the Q-QPFT are also studied. Finally, we formulate several classes of uncertainty
principles (UPs) for the Q-QPFT, which including Heisenberg-type UP,
logarithmic UP,  Hardy’s UP,  Beurling’s UP and Donoho-Stark’s
UP. It can
be regarded as the first step in the applications of the Q-QPFT in the real world.
 \\

}
\end{quotation}

\section{Introduction} \label{sec intro}
\noindent
In time-frequency analysis, the most recent signal processing  tool is the quadratic-phase Fourier transform (QPFT) introduced by Castro et al.\cite{akp1}  which provides a unified treatment of both the transient
and non-transient signals in a simple and insightful fashion.
The QPFT has five real parameters with exponential kernel. With a slight modification in \cite{akp1}, we define the QPFT as
\begin{equation}\label{qpft}
\mathcal Q_{\mu}[f](w)=\int_{\mathbb R}f(x)\Lambda_\mu(x,w)dx,
\end{equation}
where $\Lambda_\mu(x,w)$ is a quadratic-phase kernel and is given by
\begin{equation}\label{ker qpft}
\Lambda_\mu(x,w)=\sqrt{\frac{bi}{2\pi}}e^{-i(ax^2+bxw+cw^2+dx+ew)}
\end{equation}
and the corresponding inversion formula is given by
\begin{equation}\label{inv qpft}
f(x)=\int_{\mathbb R}\mathcal Q_{\mu}[f](w)\overline{\Lambda_\mu(x,w)}dw,
\end{equation}
where $a,b,c,d,e\in \mathbb R,$ $b\ne0$. These arbitrary real parameters present in (\ref{ker qpft}) are of great importance as their choice sense of rotation as well as shift can be inculcated in both the axis of time and frequency domain. Hence can be used in better analysis of non-transient signals which are employed in radar and other communication systems. Due to its global kernel and extra degrees of freedom, the QPFT has arrived  an efficient tool in solving several problems arising in diverse branches of science and
engineering, including harmonic analysis, image
processing, sampling, reproducing kernel Hilbert spaces and so on \cite{n1,n2,n3,n4}.\\

The generalization of integral transforms from real and complex numbers to the quaternion setting is popular nowadays for the study of higher dimension viz: the quaternion Fourier transform (QFT) \cite{novel17,2s20}, the quaternion linear canonical transform (QLCT) \cite{novel22,novel23}, the fractional
quaternion Fourier transform (Fr-QFT) \cite{novel33,novel34}, the quaternion offset linear canonical transform (QOLCT) \cite{o1,o2,o3,o4}. In  past decades, quaternion algebra has become a leading area of research with its
applications in color image processing, image filtering, watermarking, edge detection and
pattern recognition(see \cite{1d13,1d14,1d15,1d18,1d19,1d20,1d21}). The  Fourier transform (FT) in quaternion setting i.e.the  quaternion Fourier transform (QFT) \cite{1d22} plays a significant role in the representation of hyper-complex signals in signal processing which is believed to be the substitute of the commonly used two-
dimensional Complex Fourier Transform (CFT). The QFT has wide range of applications
see(\cite{1d23,1d24}). On the other hand the uncertainty principle (UP) plays a vital role in various scientific fields
such as mathematics, quantum physics, signal processing  and information theory \cite{novel9,novel12,novel16}. The UPs like Heisenberg’s, Hardy’s, Beurling’s associated with QFT are  given in \cite{gen1,gen3,gen4,gen5} and the extension of UPs in the domains of QLCT,QOLCT are given in \cite{s1,s2,s3,2sself,2s22}. These UPs have many applications in the analysis of optical systems, signal recovery  and so on see(\cite{novel32,novel13,novel25,novel30}). Therefore modern era of information processing is in dire need of quaternionic valued signals and therefore is a very
hot area of research.
Since the  QPFT  is a five parameter
class of linear integral transform and  has more
degrees of freedom and is more flexible than the FT, the
FRFT, the LCT but with similar computation cost as the conventional
FT. Due to the mentioned advantages, it is natural to
generalize the classical QPFT to the quaternionic algebra.

To the best of our knowledge, the generalization of the QPFT to quaternion algebra,
and the study of the properties and UPs associated with Q-QPFT have not been
carried out yet. So motivated and inspired by the merits of QPFT and QFT, we in this paper propose the novel integral transform coined as the quaternion quadratic-phase Fourier transform (Q-QPFT), which provides a unified treatment for several existing classes of signal processing tools. Therefore it
is worthwhile to rigorously study the Q-QPFT and associated UPs
which can be productive for signal processing theory and applications.

\subsection{Paper Contributions}\,\\

The contributions  of this  paper are summarized below:\\

\begin{itemize}

\item  To introduce a novel integral transform coined as the quaternion quadratic-phase Fourier transform.\\

\item To establish the fundamental relationship between the proposed transform (Q-QPFT) and the quaternion Fourier transform (QFT).\\

\item To study the fundamental properties of the proposed transform, including the parseval's formula, inversion formula, shift and modulation.\\

  \item To formulate several classes of uncertainty principles, such as the Heisenberg UP and  the logarithmic UP
associated with the quaternion quadratic-phase Fourier transform.\\

\item To formulate the Hardy's,  Beurling's and Donoho-Stark’s uncertainty principles for the Q-QPFT.\\

\end{itemize}

\subsection{Paper Outlines}\,\\
The paper is organized as follows: In Section \ref{sec 1}, we give a brief  review to the quaternion algebra and summarize some definitions and results of two-sided QFT useful in the sequel. The definition and the properties of the novel Q-QPFT
 are studied in Section \ref{sec 2}. In Section \ref{sec3}, we establish
some different forms of uncertainty principles (UPs) for the two-sided Q-QPFT which including
 Heisenberg-type UP, logarithmic UP, Hardy’s UP, Beurling’s UP, and Donoho-Stark’s UP . Finally, a conclusion is drawn in Section \ref{sec 4}.

\section{Preliminary}\label{sec 1}
In this section,  we give a brief  review to the quaternion algebra and summarize some definitions and results of two-sided QFT which will be needed throughout the paper.

\subsection{Quaternion}\,\\
Hamilton introduced the 4-D quaternion algebra in 1843 denoted by $\mathbb H$ in his honor,
\begin{equation*}
\mathbb H=\{q=[q]_0+i[q]_1+j[q]_2+k[q]_3,[q]_i\in \mathbb R,i=0,1,2,3\},
\end{equation*}
which has three imaginary units $\{i,j,k\}$ and obey Hamilton's multiplication rules: $i^2=j^2=k^2=ijk=-1,ij=-ji=k.$ Let $[q]_0$ and $\b{q}=i[q]_1+j[q]_2+k[q]_3$ be the real scalar part and the vector part of quaternion number $q\in \mathbb H.$ The conjugate of quaternion number $q\in \mathbb H$ is given by
$$q=[q]_0-i[q]_1-j[q]_2-k[q]_3$$
and its norm is defined as
$$|q|=\sqrt{q\overline{q}}=\sqrt{[q]^2_0+[q]^2_1+[q]^2_2+[q]^2_3}.$$
Also it is easy to check that
$$|pq|=|p||q|,\quad p,q\in \mathbb H.$$
Moreover the real scalar part has a cyclic multiplication symmetry
\begin{equation}\label{cyclic sym}
[pqr]_0=[qrp]_0=[rpq]_0,\quad \forall p,q,r\in \mathbb H.
\end{equation}
The inner product of quaternion functions $f,g$
 on $\mathbb R^2$ with values in $\mathbb H$ is defined as follows:

\begin{equation*}
\langle f,g\rangle=\int_{\mathbb R^2}f({\bf x})\overline{g({\bf x})}d{\bf x},\quad d{\bf x}=dx_1dx_2,
\end{equation*}
with symmetric real scalar part
\begin{equation}\label{sym int}
\langle f,g\rangle=\frac{1}{2}[(f,g)+(g,f)]=\int_{\mathbb R^2}\left[f({\bf x})\overline{g({\bf x})}\right]_0d{\bf x}.
\end{equation}
And for $f=g$, we obtain the $L^2(\mathbb R^2,\mathbb H)-$norm:
\begin{equation}
\|f\|=\left(\int_{\mathbb R^2}|f({\bf x})|^2d{\bf x}\right)^{1/2}.
\end{equation}

\subsection{Quaternion Fourier transform}\,\\
The QFT plays a vital role in signal processing and color imaging. There are three different types of QFT, the left-sided QFT, the two-sided QFT and  the right-sided QFT. Here our focus will be on two-sided QFT (in the rest of paper QFT means two-sided QFT).

 \begin{definition}[QFT \cite{gen1}]\label{qft}
The two-sided QFT of a quaternion signal $f\in L^1(\mathbb R^2,\mathbb H)$ is defined by
\begin{equation}\label{eqn qft}
\mathcal F^{\mathbb H}[f]({\bf w})=\sqrt{\frac{1}{(2\pi)^2}}\int_{\mathbb R^2}e^{-ix_1w_1}f({\bf x})e^{-jx_2w_2}d{\bf x},
\end{equation}
and corresponding inverse QFT is given by
\begin{equation}\label{eqn invqft}
f({\bf x})=\sqrt{\frac{1}{(2\pi)^2}}\int_{\mathbb R^2}e^{ix_1w_1}\mathcal F^{\mathbb H}[f]({\bf w})e^{jx_2w_2}d{\bf w},
\end{equation}
where ${\bf x}=(x_1,x_2)$ and ${\bf w}=(w_1,w_2).$
\end{definition}
\begin{lemma}[QFT Parseval \cite{2s20}]\label{lem par qft} The quaternion product of $f,g\in L^1(\mathbb R^2,\mathbb H)\cap L^2(\mathbb R^2,\mathbb H)$ and its QFT are related by
\begin{equation}
\langle f,g \rangle_{L^2(\mathbb R^2,\mathbb H)}=\langle \mathcal F^{\mathbb H}[f],\mathcal F^{\mathbb H}[g] \rangle_{L^2(\mathbb R^2,\mathbb H)}.
\end{equation}
In particular if $f=g$ we get the quaternion version of the
Plancherel formula; that is,
\begin{equation}
\|f\|^2_{L^2(\mathbb R^2,\mathbb H)}=\|\mathcal F^{\mathbb H}[f]\|_{L^2(\mathbb R^2,\mathbb H)}.
\end{equation}
\end{lemma}
\begin{lemma}\cite{2sself}\label{lem young}
If $1\le p\le 2$ and letting $\frac{1}{p}+\frac{1}{q},$ for all $f\in L^p(\mathbb R^2,\mathbb H)$, then it holds
\begin{equation}
\|\mathcal F^{\mathbb H}\|_q\le(2\pi)^{\frac{1}{q}-\frac{1}{p}}\|f\|_p.
\end{equation}
\end{lemma}

 \section{Quaternion Quadratic-Phase Fourier Transform}\label{sec 2}

In this section we extend the definition of QPFT (\ref{qpft}) to the 2D quaternionic signals and then based on the definition of the QFT (\ref{qft}), we define the novel quaternion QPFT of 2D quaternionic signals and establish several fundamental properties associated with the proposed Q-QPFT.
\subsection{QPFTs and Q-QPFT of 2D Quaternionic Signals}
\begin{definition}[QPFTs of 2D Quaternionic Signals]\label{left,right qpft}Let $\mu_s=(a_s,b_s,c_s,d_s,e_s),$ $a_s,b_s,c_s,d_s,e_s\in \mathbb R$ and $b_s\ne 0$ for $s=1,2$. Then the left-sided and right-sided QFTs of 2D quaternion signals $f\in L^2(\mathbb R^2,\mathbb H)$ are defined by
\begin{equation}\label{left}
\mathbb Q^i_{l,\mu}[f](w_1,x_2)=\int_{\mathbb R^2}\Lambda^i_{\mu_1}(x_1,w_1)f(x_1,x_2)dx_1,
\end{equation}
\begin{equation}\label{right}
\mathbb Q^j_{r,\mu}[f](x_1,w_2)=\int_{\mathbb R^2}f(x_1,x_2)\Lambda^j_{\mu_2}(x_2,w_2)dx_2,
\end{equation}
where the kernels are given by
\begin{equation}\label{kerleft qpft}
\Lambda^i_{\mu_1}(x_1,w_1)=\sqrt{\frac{b_1i}{2\pi}}e^{-i(a_1x_1^2+b_1x_1w_1+c_1w_1^2+d_1x_1+e_1w_1)},
\end{equation}
\begin{equation}\label{kerright qpft}
\Lambda^j_{\mu_2}(x_2,w_2)=\sqrt{\frac{b_2j}{2\pi}}e^{-j(a_2x_2^2+b_2x_2w_2+c_2w_2^2+d_2x_2+e_2w_2)},
\end{equation}
respectively.
\end{definition}

\begin{theorem}[Plancherel Theorem for right-sided QPFT]Let $f,g\in L^1\cap L^2(\mathbb R^2,\mathbb H),$ then
\begin{equation}\label{planc right qpft}
\langle f,g\rangle_{L^2(\mathbb R^2,\mathbb H)}=\langle \mathbb Q^j_{r}[f],\mathbb Q^j_{r}[g]\rangle_{L^2(\mathbb R^2,\mathbb H)}.
\end{equation}
And for $f=g$, we get the Parseval theorem as
\begin{equation}\label{Parsevalright qpft}
\|f\|^2_{L^2(\mathbb R^2,\mathbb H)}=\|\mathbb Q^j_{r}[f],\mathbb \|^2_{L^2(\mathbb R^2,\mathbb H)}.
\end{equation}
\end{theorem}
\begin{proof}
Proof of the above theorem follows from the one-dimensional case.
\end{proof}

Now, we introduce  the quaternionic quadratic-phase Fourier transforms (Q-QPFTs). Due to the noncommutative
property of multiplication of quaternions, there are many different
types of Q-QPFTs: two-sided Q-QPFTs, left-sided Q-QPFTs,
and right-sided Q-QPFTs.
\begin{definition}[Q-QPFTs of 2D Quaternionic Signals]\label{def Q-QPFT}
Let $\mu_s=(a_s,b_s,c_s,d_s,e_s),$\\ $a_s,b_s,c_s,d_s,e_s\in \mathbb R$ and $b_s\ne 0$ for $s=1,2$. Then the two-sided Q-QPFT, right-sided Q-QPFT and left-sided Q-QPFT of signals $f\in L^1(\mathbb R^2,\mathbb H)$ are defined by
\begin{equation}\label{eqn 2sd}
\mathbb Q^{i,j}_{T,\mu_1,\mu_2}[f]({\bf w})=\int_{\mathbb R^2}\Lambda^i_{\mu_1}(x_1,w_1)f({\bf x})\Lambda^j_{\mu_2}(x_2,w_2)d{\bf x},
\end{equation}

\begin{equation}\label{eqn rsd}
\mathbb Q^{i,j}_{R,\mu_1,\mu_2}[f]({\bf w})=\int_{\mathbb R^2}f({\bf x})\Lambda^i_{\mu_1}(x_1,w_1)\Lambda^j_{\mu_2}(x_2,w_2)d{\bf x},
\end{equation}
\begin{equation}\label{eqn lsd}
\mathbb Q^{i,j}_{L,\mu_1,\mu_2}[f]({\bf w})=\int_{\mathbb R^2}\Lambda^i_{\mu_1}(x_1,w_1)\Lambda^j_{\mu_2}(x_2,w_2)f({\bf x})d{\bf x},
\end{equation}
respectively.
Where ${\bf w}=(w_1,w_2)\in\mathbb R^2,$ ${\bf x}=(x_1,x_2)\in\mathbb R^2$ and $\Lambda^i_{\mu_1}(x_1,w_1)$ and $\Lambda^j_{\mu_2}(x_2,w_2)$ are quaternion kernel signals given by (\ref{kerleft qpft})and (\ref{kerright qpft}).
\end{definition}
The following lemma gives the relationship between
various types of Q-QPFTs.
\begin{lemma}\label{l1}
For $f\in L^1(\mathbb R^4,\mathbb H),$ the two-sided Q-QPFT of a signal $f$ can be written as sum of the two right-sided or left-sided Q-QPFT as:
\begin{equation}\label{lem 1}
\mathbb Q^{j,i}_{T,\mu_1,\mu_2}[f]=\mathbb Q^{j,i}_{L,\mu_1,\mu_2}[f_p]+\mathbb Q^{j,-i}_{L,\mu_1,\mu_2}[f_q]j,
\end{equation}
\begin{equation}\label{lem1}
\mathbb Q^{i,j}_{T,\mu_1,\mu_2}[f]=\mathbb Q^{i,j}_{R,\mu_1,\mu_2}[f_p]+\mathbb Q^{i,j}_{R,\mu_1,\mu_2}[f_q]j,
\end{equation}
where $f=f_p+f_qj,$ $f_p=f_0+if_1,$ $f_0,f_1\in \mathbb R.$
\end{lemma}
\begin{proof}
Proof of above lemma follows by  the procedure of Lemma 2.3 in \cite{sp}.
\end{proof}
\begin{remark}
The Lemma \ref{l1} assures that it is sufficient to study two-sided Q-QPFT, as the analogue
results for left-sided Q-QPFT or right-sided Q-QPFT can be deduced from (\ref{lem 1}) and (\ref{lem 1}).
\end{remark}
 \subsection{Two-sided Q-QPFT}
In this subsection we study  the two-sided Q-QPFT (for
simplicity of notation we write the Q-QPFT instead of the two-sided
Q-QPFT). First we recall the definition of Q-QPFT (\ref{eqn 2sd}) and present an example for the lucid illustration of the proposed  transform,  then we show that Q-QPFT can be reduced to the two-sided QFT. Finally, we conclude this subsection with the properties of the Q-QPFT which are vital for signal processing.
\begin{definition}[Q-QPFT]\label{def Q-QPFT}
Let $\mu_s=(a_s,b_s,c_s,d_s,e_s)$ for $s=1,2$, then the two-sided Q-QPFT of signals $f\in L^1(\mathbb R^2,\mathbb H)$ is denoted  by $\mathbb Q^{\mathbb H}_{\mu_1,\mu_2}[f]$ and defined as
\begin{equation}\label{eqn q-qpft}
\mathbb Q^{\mathbb H}_{\mu_1,\mu_2}[f]({\bf w})=\int_{\mathbb R^2}\Lambda^i_{\mu_1}(x_1,w_1)f({\bf x})\Lambda^j_{\mu_2}(x_2,w_2)d{\bf x},
\end{equation}
where ${\bf w}=(w_1,w_2)\in\mathbb R^2,$ ${\bf x}=(x_1,x_2)\in\mathbb R^2$ and $\Lambda^i_{\mu_1}(x_1,w_1)$ and $\Lambda^j_{\mu_2}(x_2,w_2)$ are quaternion kernel signals given by (\ref{kerleft qpft})and (\ref{kerright qpft}), respectively.\\
Where $a_s,b_s,c_s,d_s,e_s\in \mathbb R,b_s\ne0$ and $s=1,2.$
\end{definition}
\begin{remark} By appropriately choosing parameters in $\mu_s=(a_s,b_s,c_s,d_s,e_s),s=1,2$  the Q-QPFT(\ref{eqn q-qpft})
includes many well-known linear transforms
as special cases:
\begin{itemize}
\item For $\mu_s=(0,-1,0,0,0),s=1,2$, the Q-QPFT (\ref{eqn q-qpft}) boils down to the Quaternion-Fourier Transform \cite{gen1}.
\item As a special case, when  $\mu_s=(a_s,b_s,c_s,0,0),s=1,2$ the Q-QPFT (\ref{eqn q-qpft})can be viewed as the Quaternion Linear Canonical Transform \cite{s1}.
\item  For $\mu_s=(\cot\theta,-\csc\theta,\cot\theta,0,0),s=1,2$ the the Q-QPFT (\ref{eqn q-qpft}) leads to the
two-sided FrQFT \cite{2sself}.
\end{itemize}
\end{remark}
We now present an example for the lucid illustration of the proposed quaternion quadratic-phase Fourier transform (\ref{eqn q-qpft})
\begin{example}
Consider a 2D Gaussian
quaternionic function $f(x)=e^{-(k_1x_1^2+k_2^2)},$ with
$\quad k_1,k_2\ge 0.$ \\
\end{example}
Then by definition of Q-QPFT, we have
\begin{eqnarray*}
&&\mathbb Q^{\mathbb H}_{\mu_1,\mu_2}[f]({\bf w})\\
&&=\int_{\mathbb R^2}\sqrt{\frac{b_1i}{2\pi}}e^{-i(a_1x_1^2+b_1x_1w_1+c_1w_1^2+d_1x_1+e_1w_1)}e^{-(k_1x_1^2+k_2^2)}\\
&&\qquad\qquad\qquad \times \sqrt{\frac{b_2j}{2\pi}} e^{-j(a_2x_2^2+b_2x_2w_2+c_2w_2^2+d_2x_2+e_2w_2)}d{\bf x}\\
&&=\sqrt{\frac{b_1i}{2\pi}}\int_{\mathbb R}e^{-i(a_1x_1^2+b_1x_1w_1+c_1w_1^2+d_1x_1+e_1w_1)}e^{-k_1x_1^2}dx_1\\
&&\qquad\qquad\qquad \times \sqrt{\frac{b_2j}{2\pi}} \int_{\mathbb R}e^{-j(a_2x_2^2+b_2x_2w_2+c_2w_2^2+d_2x_2+e_2w_2)}e^{-k_2x_2^2}dx_2\\
&&=\sqrt{\frac{b_1i}{2\pi}}e^{-i(c_1w_1^2+e_1w_1)}\int_{\mathbb R}e^{-i(a_1x_1^2+b_1x_1w_1+d_1x_1)}e^{-k_1x_1^2}dx_1\\
&&\qquad\qquad\qquad \times \sqrt{\frac{b_2j}{2\pi}} e^{-j(c_2w_2^2+e_2w_2)}\int_{\mathbb R}e^{-j(a_2x_2^2+b_2x_2w_2+d_2x_2)}e^{-k_2x_2^2}dx_2\\
&&=\sqrt{\frac{b_1i}{2\pi}}e^{-i(c_1w_1^2+e_1w_1)}\int_{\mathbb R}e^{-i[a_1x_1^2+(b_1w_1+d_1)x_1]}e^{-k_1x_1^2}dx_1\\
&&\qquad\qquad\qquad \times \sqrt{\frac{b_2j}{2\pi}} e^{-j(c_2w_2^2+e_2w_2)}\int_{\mathbb R}e^{-j[a_2x_2^2+(b_2w_2+d_2)x_2]}e^{-k_2x_2^2}dx_2\\
&&=\sqrt{\frac{b_1i}{2\pi}}e^{-i(c_1w_1^2+e_1w_1)}\int_{\mathbb R}e^{-(k_1+ia_1)\left[x_1+i\frac{b_1w_1+d_1}{2(k_1+ia_1)}\right]^2}dx_1e^{-\frac{(b_1w_1+d_1)^2}{4(k_1+ia_1)}}\\
&&\qquad\qquad\qquad \times \sqrt{\frac{b_2j}{2\pi}} e^{-j(c_2w_2^2+e_2w_2)}\int_{\mathbb R}e^{-(k_2+ja_2)\left[x_2+j\frac{b_2w_2+d_2}{2(k_2+ja_2)}\right]^2}dx_2e^{-\frac{(b_2w_2+d_2)^2}{4(k_2+ja_2)}}.\\
\end{eqnarray*}
Using $\int_{\mathbb R}e^{-z(x+z')^2}dx=\sqrt{\frac{\pi}{z}},$ for $z,z'\in \mathbb C$(Gaussian integral) in above equation ,we immediately obtain
\begin{eqnarray*}
&&\mathbb Q^{\mathbb H}_{\mu_1,\mu_2}[f]({\bf w})\\
&&=\sqrt{\frac{b_1i}{2\pi}}e^{-i(c_1w_1^2+e_1w_1)}\sqrt{\frac{\pi}{k_1+ia_1}}e^{-\frac{(b_1w_1+d_1)^2}{4(k_1+ia_1)}}\\
&&\qquad\qquad\qquad \times \sqrt{\frac{b_2j}{2\pi}} e^{-j(c_2w_2^2+e_2w_2)}\sqrt{\frac{\pi}{k_2+ja_2}}e^{-\frac{(b_2w_2+d_2)^2}{4(k_2+ja_2)}}\\
&&=e^{-i(c_1w_1^2+e_1w_1)}\sqrt{\frac{b_1i}{2(k_1+ia_1)}}e^{-\frac{(b_1w_1+d_1)^2}{4(k_1+ia_1)}}\\
&&\qquad\qquad\qquad \times  e^{-j(c_2w_2^2+e_2w_2)}\sqrt{\frac{b_2j}{2(k_2+ja_2)}}e^{-\frac{(b_2w_2+d_2)^2}{4(k_2+ja_2)}}.\\
\end{eqnarray*}

Now we gave the fundamental relationship between the proposed Q-QPFT and the QFT.
\begin{theorem}\label{relation} The Q-QPFT (\ref{eqn q-qpft}) of a quaternion signal $f\in L^1(\mathbb R^2,\mathbb H)$ can be reduced to the QFT
\begin{equation}\label{qft1}
\mathcal F^{\mathbb H}[G_f]({\bf w})=\frac{1}{\sqrt{(2\pi)^2}}\int_{\mathbf R^2}e^{-ix_1w_1}G_f({\bf x})e^{-jx_2w_2}d{\bf x},
\end{equation}
where
\begin{equation*}
\mathcal F^{\mathbb H}[G_f]({\bf w})=F\left(\frac{{\bf w}}{{\bf b}}\right),
\end{equation*}
\begin{equation}\label{func hf}
G_f({\bf x})=\sqrt{b_1i}\tilde f({\bf x})\sqrt{b_2j},
\end{equation}
\begin{equation*}
\tilde f({\bf x})=e^{-i(a_1x_1^2+d_1x_1)}f({\bf x})e^{-j(a_2x_2^2+d_2x_2)}
\end{equation*}
with \begin{equation*}
F({\bf w})=\frac{1}{\sqrt{(2\pi)^2}}\int_{\mathbf R^2}e^{-ix_1b_1w_1}G_f({\bf x})e^{-jx_2b_2w_2}d{\bf x},
\end{equation*}
\begin{equation}\label{fun F}
F({\bf w})=e^{i(c_1w_1^2+e_1w_1)}\mathbb Q^{\mathbb H}_{\mu_1,\mu_2}[f]({\bf w})e^{j(c_2w_2^2+e_2w_2)}
\end{equation}
\end{theorem}
\begin{proof}
From Definition \ref{def Q-QPFT}, we obtain
\begin{eqnarray*}
&&\mathbb Q^{\mathbb H}_{\mu_1,\mu_2}[f]({\bf w})\\
&&=\int_{\mathbb R^2}\sqrt{\frac{b_1i}{2\pi}}e^{-i(a_1x_1^2+b_1x_1w_1+c_1w_1^2+d_1x_1+e_1w_1)}f({\bf x})\\
&&\qquad\qquad\qquad \times \sqrt{\frac{b_2j}{2\pi}} e^{-j(a_2x_2^2+b_2x_2w_2+c_2w_2^2+d_2x_2+e_2w_2)}d{\bf x}\\
&&=\sqrt{\frac{b_1i}{2\pi}}e^{-i(c_1w_1^2+e_1w_1)}\int_{\mathbb R^2}e^{-ix_1b_1w_1}\tilde f({\bf x})d{\bf x}e^{-jx_2b_2w_2}\sqrt{\frac{b_2j}{2\pi}}e^{-j(c_2w_2^2+e_2w_2)}
\end{eqnarray*}
Then, multiplying both sides of the above equation by $e^{i(c_1w_1^2+e_1w_1)}e^{j(c_2w_2^2+e_2w_2)},$ yields
\begin{eqnarray*}
&&e^{i(c_1w_1^2+e_1w_1)}\mathbb Q^{\mathbb H}_{\mu_1,\mu_2}[f]({\bf w})e^{j(c_2w_2^2+e_2w_2)}\\
&&=\sqrt{\frac{b_1i}{2\pi}}\int_{\mathbb R^2}e^{-ix_1b_1w_1}\tilde f({\bf x})d{\bf x}e^{-jx_2b_2w_2}\sqrt{\frac{b_2j}{2\pi}}\\
&&={\frac{1}{\sqrt{(2\pi)^2}}}\int_{\mathbb R^2}e^{-ix_1b_1w_1}G_f({\bf x})d{\bf x}e^{-jx_2b_2w_2}{\frac{1}{\sqrt{(2\pi)^2}}}\\
&&=\mathcal F^{\mathbb H}[G_f](\bf{bw}),
\end{eqnarray*}
where $\bf{bw}=(b_1w_1,b_2w_2)$. This leads to the desired result.

\end{proof}
\begin{theorem}[Inversion formula]Let $\mathbb Q^{\mathbb H}_{\mu_1,\mu_2}[f]\in L^1(\mathbb R^2,\mathbb H), $ then every signal $f\in L^1(\mathbb R^2,\mathbb H)$ can be reconstructed back by the formula
\begin{equation}
f({\bf x})=\int_{\mathbb R^2}\overline{\Lambda(x_1,w_1)}\mathbb Q^{\mathbb H}_{\mu_1,\mu_2}[f](w)\overline{\Lambda(x_2,w_2)}d{\bf w}.
\end{equation}
\end{theorem}
\begin{proof}
By the application of the inversion formula of QFT, we have
\begin{eqnarray*}
G_f({\bf x})
&=&\frac{1}{\sqrt{(2\pi)^2}}\int_{\mathbf R^2}e^{ix_1w_1}\mathcal F^{\mathbb H}[G_f]({\bf w})e^{jx_2w_2}d{\bf w}\\
&=&\frac{1}{\sqrt{(2\pi)^2}}\int_{\mathbf R^2}e^{ix_1w_1}F\left({\bf\frac{w} {b}}\right)e^{jx_2w_2}d{\bf w}\\
\end{eqnarray*}
On setting $\bf{w=bw},$ above equation yields
\begin{eqnarray}
\label{26}G_f({\bf x})&=&\frac{b_1b_2}{\sqrt{(2\pi)^2}}\int_{\mathbf R^2}e^{ix_1b_1w_1}F({\bf w})e^{jx_2b_2w_2}d{\bf w}\\
\nonumber &=&\frac{b_1b_2}{\sqrt{(2\pi)^2}}\int_{\mathbf R^2}e^{ix_1b_1w_1}e^{i(c_1w_1^2+e_1w_1)}\mathbb Q^{\mathbb H}_{\mu_1,\mu_2}[f]({\bf w})e^{j(c_2w_2^2+e_2w_2)}e^{jx_2b_2w_2}d{\bf w}.
\end{eqnarray}
Which implies
\begin{eqnarray}
\nonumber&&\sqrt{b_1i}e^{-i(a_1x_1^2+d_1x_1)}f({\bf x})e^{-j(a_2x_2^2+d_2x_2)}\sqrt{b_2j}\\
\label{27}&&=\frac{b_1b_2}{\sqrt{(2\pi)^2}}\int_{\mathbf R^2}e^{i(b_1x_1w_1+c_1w_1^2+e_1w_1)}\mathbb Q^{\mathbb H}_{\mu_1,\mu_2}[f]({\bf w})e^{j(b_2x_2w_2+c_2w_2^2+e_2w_2)}d{\bf w}.
\end{eqnarray}
On further simplifying, we have
\begin{eqnarray*}
&&f({\bf x})\\
&&=\int_{\mathbf R^2}\sqrt{\frac{b_1}{{2\pi i}}}e^{i(a_1x_1^2+b_1x_1w_1+c_1w_1^2+d_1x_1+e_1w_1)}\mathbb Q^{\mathbb H}_{\mu_1,\mu_2}[f]({\bf w})\\
&&\qquad\qquad\qquad\times\sqrt{\frac{b_2}{{2\pi j}}}e^{j(a_2x_2^2b_2x_2w_2+c_2w_2^2+d_2x_2+e_2w_2)}d{\bf w}\\
&&=\int_{\mathbb R^2}\overline{\Lambda(x_1,w_1)}\mathbb Q^{\mathbb H}_{\mu_1,\mu_2}[f](w)\overline{\Lambda(x_2,w_2)}d{\bf w}.
\end{eqnarray*}
Which completes the proof.
\end{proof}
\begin{theorem}[Parseval's formula] Let $f,g\in L^1(\mathbb R^2,\mathbb H)\cap L^2(\mathbb R^2,\mathbb H),$ be two quaternion signals, then we have
\begin{equation}\label{29}
\langle f,g\rangle_{L^2(\mathbb R^2,\mathbb H)}=\left\langle\mathbb Q^{\mathbb H}_{\mu_1,\mu_2}[f],\mathbb Q^{\mathbb H}_{\mu_1,\mu_2}[g]\right\rangle_{L^2(\mathbb R^2,\mathbb H)}.
\end{equation}
For $f=g,$ we have
\begin{equation}\label{30}
\|f\|^2_{L^2(\mathbb R^2,\mathbb H)}=\|\mathbb Q^{\mathbb H}_{\mu_1,\mu_2}[f]\|^2_{L^2(\mathbb R^2,\mathbb H)}.
\end{equation}
\end{theorem}

\begin{proof}
By the Parseval's formula for the QFT and (\ref{cyclic sym}), we have
\begin{eqnarray}
\nonumber\langle G_f,G_g\rangle&=&\langle\mathcal F^{\mathbb H}[ G_f],\mathcal F^{\mathbb H}[ G_g]\rangle\\
\nonumber&=&\left[\int_{\mathbb R^2}\mathcal F^{\mathbb H}[ G_f]({\bf w})\overline{\mathcal F^{\mathbb H}[ G_g]({\bf w})}d{\bf w}\right]_0\\
\nonumber&=&|b_1b_2|\left[\int_{\mathbb R^2}\mathcal F^{\mathbb H}[ G_f]({\bf bw})\overline{\mathcal F^{\mathbb H}[ G_g]({\bf bw})}d{\bf w}\right]_0\\
\nonumber&=&|b_1b_2|\left[\int_{\mathbb R^2}e^{i(c_1w_1^2+e_1w_1)}\mathbb Q^{\mathbb H}_{\mu_1,\mu_2}[f]({\bf w})e^{j(c_2w_2^2+e_2w_2)}\right.\\
\nonumber&&\qquad\qquad\qquad\qquad\times\left.\overline{e^{i(c_1w_1^2+e_1w_1)}\mathbb Q^{\mathbb H}_{\mu_1,\mu_2}[g]({\bf w})e^{j(c_2w_2^2+e_2w_2)}}d{\bf w}\right]_0\\
\label{p1}&=&|b_1b_2|\left[\int_{\mathbb R^2}\mathbb Q^{\mathbb H}_{\mu_1,\mu_2}[f]({\bf w})\overline{\mathbb Q^{\mathbb H}_{\mu_1,\mu_2}[g]({\bf w})}d{\bf w}\right]_0.
\end{eqnarray}
Again, we have
\begin{eqnarray}
\nonumber\langle G_f,G_g\rangle&=&\left[\int_{\mathbb R^2}G_f({\bf x})\overline{G_g({\bf x})}d{\bf x}\right]_0\\
\nonumber&=&\left[\int_{\mathbb R^2}\sqrt{b_1i}\tilde f({\bf x})\sqrt{b_2j}\overline{\sqrt{b_1i}\tilde g({\bf x})\sqrt{b_2j}}d{\bf x}\right]_0\\
\nonumber&=&\left[\int_{\mathbb R^2}|b_1b_2|\tilde f({\bf x})\overline{\tilde g({\bf x})}d{\bf x}\right]_0\\
\nonumber&=&|b_1b_2|\left[\int_{\mathbb R^2}e^{-i(a_1x_1^2+d_1x_1)}f({\bf x})e^{j(a_2x_2^2+d_2x_2)}\overline{e^{-i(a_1x_1^2+d_1x_1)} g({\bf x})e^{j(a_2x_2^2+d_2x_2)}}d{\bf x}\right]_0\\
\label{p2}&=&|b_1b_2|\left[\int_{\mathbb R^2}f({\bf x})\overline{ g({\bf x})}d{\bf x}\right]_0.
\end{eqnarray}
On comparing (\ref{p1}) and  (\ref{p2}), we get the desired the result.\\
\end{proof}
\begin{theorem}[Linearity property]Let $f,g\in L^2(\mathbb R^2,\mathbb H),$ then Q-QPFT is a linear operator namely
\begin{equation}
\mathbb Q^{\mathbb H}_{\mu_1,\mu_2}[\alpha f+\beta g]({\bf w})=\mathbb Q^{\mathbb H}_{\mu_1,\mu_2}[\alpha f]({\bf w})+\mathbb Q^{\mathbb H}_{\mu_1,\mu_2}[\beta g]({\bf w}),
\end{equation}
for arbitrary real constants $\alpha$ and $\beta$.
\end{theorem}
\begin{proof}
We omit proof as it follows from Definition \ref{def Q-QPFT}.
\end{proof}

\begin{theorem}[Shift property]For any  quaternion signal $f\in L^2(\mathbb R^2,\mathbb H)$  and for ${\bf k}\in \mathbb R^2,$ we have
\begin{eqnarray}
\nonumber&&\mathbb Q^{\mathbb H}_{\mu_1,\mu_2}[f(\bf{x-k})]({\bf w})\\
\nonumber&&=e^{-i\left[a_1k_1^2+d_1k_1+b_1k_1w_1-4\frac{a^2_1c_1}{b_1^2}k_1^2-4\frac{a_1}{b_1}c_1w_1k_1 -2\frac{a_1}{b_1}k_1 \right]}\mathbb Q^{\mathbb H}_{\mu_1,\mu_2}[f]\left({\bf w+2\frac{a}{b}k}\right)\\
&&\times e^{-j\left[a_2k_2^2+d_2k_2+b_2k_2w_2-4\frac{a^2_2c_2}{b_2^2}k_2^2-4\frac{a_2}{b_2}c_2w_2k_2 -2\frac{a_2}{b_2}k_2 \right]}
\end{eqnarray}
\end{theorem}
\begin{proof}
We have from (\ref{eqn q-qpft})
\begin{eqnarray*}
&&\mathbb Q^{\mathbb H}_{\mu_1,\mu_2}[f(\bf{x-k})]({\bf w})\\
&&=\int_{\mathbb R^2}\sqrt{\frac{b_1i}{2\pi}}e^{-i(a_1x_1^2+b_1x_1w_1+c_1w_1^2+d_1x_1+e_1w_1)}f(\bf{x-k})\\
&&\qquad\qquad\qquad \times \sqrt{\frac{b_2j}{2\pi}} e^{-j(a_2x_2^2+b_2x_2w_2+c_2w_2^2+d_2x_2+e_2w_2)}d{\bf x}
\end{eqnarray*}
By making the change of a variable $\bf{x-k=y}$, above equation yields
\begin{eqnarray*}
&&\mathbb Q^{\mathbb H}_{\mu_1,\mu_2}[f(\bf{x-k})]({\bf w})\\
&&=\int_{\mathbb R^2}\sqrt{\frac{b_1i}{2\pi}}e^{-i[a_1(y_1+k_1)^2+b_1(y_1+k_1)w_1+c_1w_1^2+d_1(y_1+k_1)+e_1w_1]}f({\bf y})\\
&&\qquad\qquad\qquad \times \sqrt{\frac{b_2j}{2\pi}} e^{-j[a_2(y_2+k_2)^2+b_2(y_2+k_2)w_2+c_2w_2^2+d_2(y_2+k_2)+e_2w_2)}d{\bf y}\\
&&=\int_{\mathbb R^2}\sqrt{\frac{b_1i}{2\pi}}e^{-i\left[a_1y_1^2+b_1\left(w_1+2\frac{a_1}{b_1}k_1\right)y_1+c_1\left(w_1+2\frac{a_1}{b_1}k_1\right)^2+d_1y_1+e_1\left(w_1+2\frac{a_1}{b_1}k_1\right)\right]}\\
&&\qquad\qquad\times e^{-i\left[a_1k_1^2+d_1k_1+b_1k_1w_1-4\frac{a^2_1c_1}{b_1^2}k_1^2-4\frac{a_1}{b_1}c_1w_1k_1 -2\frac{a_1}{b_1}k_1 \right]}f({\bf y})\\
&&\qquad\qquad\times\sqrt{\frac{b_2j}{2\pi}}e^{-j\left[a_2y_2^2+b_2\left(w_2+2\frac{a_2}{b_2}k_2\right)y_2+c_2\left(w_2+2\frac{a_2}{b_2}k_2\right)^2+d_2y_2+e_2\left(w_2+2\frac{a_2}{b_2}k_2\right)\right]}\\
&&\qquad\qquad\times e^{-j\left[a_2k_2^2+d_2k_2+b_2k_2w_2-4\frac{a^2_2c_2}{b_2^2}k_2^2-4\frac{a_2}{b_2}c_2w_2k_2 -2\frac{a_2}{b_2}k_2 \right]}d{\bf y}\\\\
&&=e^{-i\left[a_1k_1^2+d_1k_1+b_1k_1w_1-4\frac{a^2_1c_1}{b_1^2}k_1^2-4\frac{a_1}{b_1}c_1w_1k_1 -2\frac{a_1}{b_1}k_1 \right]}\\
&&\qquad\times \int_{\mathbb R^2}\sqrt{\frac{b_1i}{2\pi}}e^{-i\left[a_1y_1^2+b_1\left(w_1+2\frac{a_1}{b_1}k_1\right)y_1+c_1\left(w_1+2\frac{a_1}{b_1}k_1\right)^2+d_1y_1+e_1\left(w_1+2\frac{a_1}{b_1}k_1\right)\right]}f({\bf y})\\
&&\qquad\qquad\times\sqrt{\frac{b_2j}{2\pi}}e^{-j\left[a_2y_2^2+b_2\left(w_2+2\frac{a_2}{b_2}k_2\right)y_2+c_2\left(w_2+2\frac{a_2}{b_2}k_2\right)^2+d_2y_2+e_2\left(w_2+2\frac{a_2}{b_2}k_2\right)\right]}d{\bf y}\\
&&\qquad\qquad\times e^{-j\left[a_2k_2^2+d_2k_2+b_2k_2w_2-4\frac{a^2_2c_2}{b_2^2}k_2^2-4\frac{a_2}{b_2}c_2w_2k_2 -2\frac{a_2}{b_2}k_2 \right]}\\\\
&&=e^{-i\left[a_1k_1^2+d_1k_1+b_1k_1w_1-4\frac{a^2_1c_1}{b_1^2}k_1^2-4\frac{a_1}{b_1}c_1w_1k_1 -2\frac{a_1}{b_1}k_1 \right]}\\
&&\qquad\times\mathbb Q^{\mathbb H}_{\mu_1,\mu_2}[f]\left(w_1+2\frac{a_1}{b_1}k_1,w_2+2\frac{a_2}{b_2}k_2\right)\\
&&\times e^{-j\left[a_2k_2^2+d_2k_2+b_2k_2w_2-4\frac{a^2_2c_2}{b_2^2}k_2^2-4\frac{a_2}{b_2}c_2w_2k_2 -2\frac{a_2}{b_2}k_2 \right]}.
\end{eqnarray*}
Which completes the proof.
\end{proof}

\begin{theorem}[Modulation property]The quaternion quadratic-phase Fourier transform (\ref{eqn q-qpft}) of a modulated  quaternion signal $\mathcal M_{\bf w_0}f({\bf x})=e^{ix_1u_0}f({\bf x})e^{jx_2v_0}, {\bf w_0}=(u_0,v_0)$ is given by
\begin{eqnarray}
\nonumber\mathbb Q^{\mathbb H}_{\mu_1,\mu_2}[\mathcal M_{\bf w_0}f]({\bf w})&=&e^{i\left[\frac{c_1{u_0^2}-2b_1c_1u_0w_1-b_1e_1u_0}{b_1^2}\right]}\mathbb Q^{\mathbb H}_{\mu_1,\mu_2}[f]\left({\bf w-\frac{w_0}{b}}\right)\\
&&\qquad\qquad\times e^{j\left[\frac{c_2{v_0^2}-2b_2c_2v_0w_2-b_2e_2v_0}{b_2^2}\right]}.
\end{eqnarray}
\end{theorem}
\begin{proof}
From Definition \ref{def Q-QPFT}, we get
\begin{eqnarray*}
\mathbb Q^{\mathbb H}_{\mu_1,\mu_2}[\mathcal M_{\bf w_0}f]({\bf w})
&=&\int_{\mathbb R^2}\sqrt{\frac{b_1i}{2\pi}}e^{-i(a_1x_1^2+b_1x_1w_1+c_1w_1^2+d_1x_1+e_1w_1)}e^{ix_1u_0}f({\bf x})e^{jx_2v_0}\\
&&\qquad \times \sqrt{\frac{b_2j}{2\pi}} e^{-j(a_2x_2^2+b_2x_2w_2+c_2w_2^2+d_2x_2+e_2w_2)}d{\bf x}\\
&=&\int_{\mathbb R^2}\sqrt{\frac{b_1i}{2\pi}}e^{-i\left[a_1x_1^2+b_1x_1\left(w_1-\frac{u_0}{b_1}\right)+c_1\left(w_1-\frac{u_0}{b_1}\right)^2+d_1x_1+e_1\left(w_1-\frac{u_0}{b_1}\right)\right]}\\
&&\qquad\times e^{i\left[c_1\frac{u_0^2}{b_1^2}-2c_1\frac{u_0}{b_1}w_1-e_1\frac{u_0}{b_1}\right]}f({\bf x})e^{j\left[c_2\frac{u_0^2}{b_2^2}-2c_2\frac{u_0}{b_2}w_2-e_2\frac{u_0}{b_2}\right]}\\
&&\qquad \sqrt{\frac{b_2j}{2\pi}} e^{-j\left[a_2x_2^2+b_2x_2\left(w_2-\frac{u_0}{b_2}\right)+c_2\left(w_2-\frac{u_0}{b_2}\right)^2+d_2x_2+e_2\left(w_2-\frac{u_0}{b_2}\right)\right]}d{\bf x}\\
&=&e^{i\left[c_1\frac{u_0^2}{b_1^2}-2c_1\frac{u_0}{b_1}w_1-e_1\frac{u_0}{b_1}\right]}\\
&&\quad\times\int_{\mathbb R^2}\sqrt{\frac{b_1i}{2\pi}}e^{-i\left[a_1x_1^2+b_1x_1\left(w_1-\frac{u_0}{b_1}\right)+c_1\left(w_1-\frac{u_0}{b_1}\right)^2+d_1x_1+e_1\left(w_1-\frac{u_0}{b_1}\right)\right]}f({\bf x})\\
&&\quad\times \sqrt{\frac{b_2j}{2\pi}} e^{-j\left[a_2x_2^2+b_2x_2\left(w_2-\frac{u_0}{b_2}\right)+c_2\left(w_2-\frac{u_0}{b_2}\right)^2+d_2x_2+e_2\left(w_2-\frac{u_0}{b_2}\right)\right]}d{\bf x}\\
&&\qquad\qquad\times e^{j\left[c_2\frac{u_0^2}{b_2^2}-2c_2\frac{u_0}{b_2}w_2-e_2\frac{u_0}{b_2}\right]}\\
&=&e^{i\left[c_1\frac{u_0^2}{b_1^2}-2c_1\frac{u_0}{b_1}w_1-e_1\frac{u_0}{b_1}\right]}\mathbb Q^{\mathbb H}_{\mu_1,\mu_2}[f]\left(w_1-\frac{u_0}{b_1},w_2-\frac{v_0}{b_2}\right)\\
&&\qquad\qquad\times e^{j\left[c_2\frac{u_0^2}{b_2^2}-2c_2\frac{u_0}{b_2}w_2-e_2\frac{u_0}{b_2}\right]}.
\end{eqnarray*}
Which completes the proof.
\end{proof}

We omit properties like Reflection, Conjugation and Scaling as they directly follows from Definition \ref{def Q-QPFT}.
\begin{theorem}[Hausdorff-Young]Let $1\le p\le 2$ and $\frac{1}{p}+\frac{1}{q}=1,$ then for all $f\in L^p(\mathbb R^2,\mathbb H)$ following inequality holds
\begin{eqnarray}
\|\mathbb Q^{\mathbb H}_{\mu_1,\mu_2}[f]\|_q&\le&(2\pi)^{\frac{1}{q}-\frac{1}{p}}|b_1b_2|^{\frac{1}{2}-\frac{1}{q}}\|f({\bf x})\|_p .
\end{eqnarray}
\end{theorem}

\begin{proof}
From Lemma \ref{lem young}, we have
\begin{eqnarray*}
\|\mathcal F^{\mathbb H}[f]({\bf w})\|_q&\le&(2\pi)^{\frac{1}{q}-\frac{1}{p}}\|f({\bf x})\|_p .
\end{eqnarray*}
Replacing $f$ by $G_f$, we have from above equation
\begin{eqnarray*}
\|\mathcal F^{\mathbb H}[G_f]({\bf w})\|_q&\le&(2\pi)^{\frac{1}{q}-\frac{1}{p}}\|G_f({\bf x})\|_p .
\end{eqnarray*}
With the help of equations present in Theorem \ref{relation}, above yields
\begin{eqnarray*}
\left\| F\left({\bf\frac {w}{b}}\right)\right\|_q&\le&(2\pi)^{\frac{1}{q}-\frac{1}{p}}\|\sqrt{b_1 i}\tilde f({\bf x})\sqrt{b_2 j}\|_p .
\end{eqnarray*}
On substituting ${\bf w=bw},$ we get
\begin{eqnarray*}
|b_1b_2|^{\frac{1}{q}}\left\| F({\bf w})\right\|_q&\le&(2\pi)^{\frac{1}{q}-\frac{1}{p}}\sqrt{b_1b_2}\|\tilde f({\bf x})\|_p .
\end{eqnarray*}
Again using Theorem \ref{relation}, we obtain
\begin{eqnarray*}
|b_1b_2|^{\frac{1}{q}}\left\| \mathbb Q^{\mathbb H}_{\mu_1,\mu_2}[f]({\bf w})\right\|_q&\le&(2\pi)^{\frac{1}{q}-\frac{1}{p}}\sqrt{b_1b_2}\| f({\bf x})\|_p .
\end{eqnarray*}
Further simplification yields
\begin{eqnarray*}
\left\| \mathbb Q^{\mathbb H}_{\mu_1,\mu_2}[f]({\bf w})\right\|_q&\le&(2\pi)^{\frac{1}{q}-\frac{1}{p}}|b_1b_2|^{\frac{1}{2}-\frac{1}{q}}\| f({\bf x})\|_p .
\end{eqnarray*}
Which completes the proof.
\end{proof}
\section{Uncertainty Principles Associated with the Quaternion-QPFT}\label{sec3}
In this section based on the fundamental relationship between Q-QPFT and QFT, we investigate some different forms of UPs associated with Q-QPFT including   Heisenberg UP, logarithmic UPs, Hardy's UP, Beurling’s UP, and Donoho-Stark’s UP.\\

Lets begin with the Heisenberg type uncertainty principle for the proposed transform (Q-QPFT), which is a generalization of the corresponding Heisenberg's uncertainty principle for the QFT.
\begin{theorem}[Heisenberg UP for the Q-QPFT]Let $\mathbb Q^{\mathbb H}_{\mu_1,\mu_2}[f]$  be the quaternion quadratic-phase Fourier transform of signal $f,$ then for  $f\in L^1(\mathbb R^2,\mathbb H)\cap L^2(\mathbb R^2,\mathbb H)$, $\partial f/\partial x_s \in L^2(\mathbb R^2,\mathbb H)$ and $\mathbb Q^{\mathbb H}_{\mu_1,\mu_2}[f],w_s\mathbb Q^{\mathbb H}_{\mu_1,\mu_2}[f]\in L^2(\mathbb R^2,\mathbb H), s=1,2.$ The following inequality holds:
\begin{equation}
\int_{\mathbb R^2} x_s^2\left| f({\bf x})\right|^2d{\bf x}\int_{\mathbb R^2}w_s^2|\mathbb Q^{\mathbb H}_{\mu_1,\mu_2}[f]({\bf  w})|^2d{\bf w}
\ge\frac{1}{4b_s^2}\left(\int_{\mathbb R^2}\left| f({\bf x})\right|^2\right)^2, s=1,2.
\end{equation}
\end{theorem}
\begin{proof}
The classical Heisenberg uncertainty principle in the QFT domain is given by \cite{sp12}
\begin{eqnarray}\label{hsb qft}
\int_{\mathbb R^2} x_s^2|f({\bf x})|^2d{\bf x}\int_{\mathbb R^2}w_s^2|\mathcal F^{\mathbb H}[f]({\bf w})|^2d{\bf w}\ge\frac{1}{4}\left(\int_{\mathbb R^2}|f({\bf x})|^2\right)^2, \quad s=1,2.
\end{eqnarray}
Replacing $f$ by $G_f$ in (\ref{hsb qft}), we have
\begin{eqnarray}\label{hsb a}
\int_{\mathbb R^2} x_s^2|G_f({\bf x})|^2d{\bf x}\int_{\mathbb R^2}w_s^2|\mathcal F^{\mathbb H}[G_f]({\bf w})|^2d{\bf w}\ge\frac{1}{4}\left(\int_{\mathbb R^2}|G_f({\bf x})|^2\right)^2.
\end{eqnarray}
On substituting $\bf{w=bw},$ (\ref{hsb a}) yields
\begin{eqnarray}\label{hsb b}
\int_{\mathbb R^2} x_s^2|G_f({\bf x})|^2d{\bf x}\int_{\mathbb R^2}b_s^2w_s^2|b_1b_2||\mathcal F^{\mathbb H}[G_f]({\bf b w})|^2d{\bf w}\ge\frac{1}{4}\left(\int_{\mathbb R^2}|G_f({\bf x})|^2\right)^2.
\end{eqnarray}
Using equations present in Theorem \ref{relation} in (\ref{hsb b}), we get
\begin{eqnarray}
\nonumber&&\int_{\mathbb R^2} x_s^2|\sqrt{b_1i}\tilde f({\bf x})\sqrt{b_2j}|^2d{\bf x}\int_{\mathbb R^2}b_s^2w_s^2|b_1b_2||e^{i(c_1w_1^2+e_1w_1)}\mathbb Q^{\mathbb H}_{\mu_1,\mu_2}[f]({\bf  w})e^{j(c_2w_2^2+e_2w_2)}|^2d{\bf w}\\
\label{hsb c}&&\ge\frac{1}{4}\left(\int_{\mathbb R^2}|\sqrt{b_1i}\tilde f({\bf x})\sqrt{b_2j}|^2\right)^2,
\end{eqnarray}
Which implies
\begin{eqnarray*}\label{hsb d}
\int_{\mathbb R^2}|b_1b_2| x_s^2|\tilde f({\bf x})|^2d{\bf x}\int_{\mathbb R^2}b_s^2w_s^2|b_1b_2||\mathbb Q^{\mathbb H}_{\mu_1,\mu_2}[f]({\bf  w})|^2d{\bf w}
&\ge&\frac{|b_1b_2|^2}{4}\left(\int_{\mathbb R^2}|\tilde f({\bf x})|^2\right)^2.
\end{eqnarray*}
Hence,
\begin{eqnarray*}
\nonumber&&|b_1b_2|^2\int_{\mathbb R^2} x_s^2\left|e^{-i(a_1x_1^2+d_1x_1)} f({\bf x})e^{-j(a_2x_2^2+d_2x_2)}\right|^2d{\bf x}\int_{\mathbb R^2}b_s^2w_s^2|\mathbb Q^{\mathbb H}_{\mu_1,\mu_2}[f]({\bf  w})|^2d{\bf w}\\
\label{hsb e}&&\ge\frac{|b_1b_2|^2}{4}\left(\int_{\mathbb R^2}\left|e^{-i(a_1x_1^2+d_1x_1)} f({\bf x})e^{-j(a_2x_2^2+d_2x_2)}\right|^2\right)^2.
\end{eqnarray*}
Equivalently
\begin{eqnarray}\label{hsb f}
|b_1b_2|^2\int_{\mathbb R^2} x_s^2\left| f({\bf x})\right|^2d{\bf x}\int_{\mathbb R^2}b_s^2w_s^2|\mathbb Q^{\mathbb H}_{\mu_1,\mu_2}[f]({\bf  w})|^2d{\bf w}
&\ge&\frac{|b_1b_2|^2}{4}\left(\int_{\mathbb R^2}\left| f({\bf x})\right|^2\right)^2.
\end{eqnarray}
Simplifying (\ref{hsb f}), we obtain
\begin{eqnarray*}\label{hsb f
g}
\int_{\mathbb R^2} x_s^2\left| f({\bf x})\right|^2d{\bf x}\int_{\mathbb R^2}w_s^2|\mathbb Q^{\mathbb H}_{\mu_1,\mu_2}[f]({\bf  w})|^2d{\bf w}
&\ge&\frac{1}{4b_s^2}\left(\int_{\mathbb R^2}\left| f({\bf x})\right|^2\right)^2.
\end{eqnarray*}
Which completes the proof.\\
\end{proof}

The directional uncertainty
principle for the Q-QPFT takes the following form

\begin{theorem}Let $f\in L^1(\mathbb R^2,\mathbb H)\cap L^2(\mathbb R^2,\mathbb H)$ and for $\mathbb Q^{\mathbb H}_{\mu_1,\mu_2}[f],|{\bf w}|^2\mathbb Q^{\mathbb H}_{\mu_1,\mu_2}[f]\in L^2(\mathbb R^2,\mathbb H),$ we have, the following inequality:
\begin{equation}\label{d1}
\int_{\mathbb R^2} |{\bf x}|^2\left| f({\bf x})\right|^2d{\bf x}\int_{\mathbb R^2}|{\bf w}|^2|\mathbb Q^{\mathbb H}_{\mu_1,\mu_2}[f]({\bf  w})|^2d{\bf w}
\ge\frac{1}{|{\bf b}|^2}\left(\int_{\mathbb R^2}\left| f({\bf x})\right|^2\right)^2.
\end{equation}
\end{theorem}
\begin{proof}
The directional uncertainty principle in the QFT domain reads [Theorem 16 \cite{sp}]
\begin{equation}\label{qq}
\int_{\mathbb R^2} |{\bf x}|^2\left| f({\bf x})\right|^2d{\bf x}\int_{\mathbb R^2}|{\bf w}|^2|\mathcal F^{\mathbb H}[f]({\bf  w})|^2d{\bf w}
\ge\left(\int_{\mathbb R^2}\left| f({\bf x})\right|^2\right)^2.
\end{equation}
Now using the machinery of previous theorem in (\ref{qq}), we will get the desired result (\ref{d1}).\\
\end{proof}

Next, using Logarithmic uncertainty
principle for the QFT, we establish Logarithmic uncertainty principle for the proposed Q-QPFT.

\begin{theorem}[Logarithmic UP for the Q-QPFT] Let $\mathbb Q^{\mathbb H}_{\mu_1,\mu_2}[f]$  be the quaternion quadratic-phase Fourier transform of signal $f\in \mathcal S(\mathbb R^2,\mathbb H)$[Schwartz space]. Then we have the following logarithmic inequality
\begin{eqnarray}
\int_{\mathbb R^2}\ln|{\bf x}|\left| f({\bf x})\right|^2d{\bf x}+\int_{\mathbb R^2}\ln|{\bf w}|\left|\mathbb Q^{\mathbb H}_{\mu_1,\mu_2}[f]({\bf w})\right|^2d{\bf w}
\ge (D-\ln|{\bf b}|)\int_{\mathbb R^2}\left| f({\bf x})\right|^2d{\bf x}.
\end{eqnarray}
\end{theorem}

\begin{proof}
For any $f\in \mathcal S(\mathbb R^2,\mathbb H)$, the logarithmic uncertainty principle for the two-sided quaternion Fourier
transform reads [Lemma 3.1 \cite{2sself}]
\begin{equation}\label{log1}
\int_{\mathbb R^2}\ln|{\bf x}||f({\bf x})|^2d{\bf x}+\int_{\mathbb R^2}\ln|{\bf w}|\left|\mathcal F^{\mathbb H}[f]({\bf x})\right|^2d{\bf w}\ge D\int_{\mathbb R^2}|f({\bf x})|^2d{\bf x},
\end{equation}
where $D=\ln(2\pi^2)-2\psi(1/2),$ $\psi=\frac{d}{dt}(\ln(\Gamma(x)))$ and $\Gamma(x)$ is a Gamma function.\\
Replacing $f$ by $G_f$ defined in Theorem \ref{relation} on both sides of (\ref{l1}), we have
 \begin{equation}\label{log2}
\int_{\mathbb R^2}\ln|{\bf x}||G_f({\bf x})|^2d{\bf x}+\int_{\mathbb R^2}\ln|{\bf w}|\left|\mathcal F^{\mathbb H}[G_f]({\bf w})\right|^2d{\bf w}\ge D\int_{\mathbb R^2}|G_f({\bf x})|^2d{\bf x}.
\end{equation}
On substituting $\bf{w=bw},$ (\ref{log2}) yields
\begin{eqnarray}\label{log3}
\int_{\mathbb R^2} \ln|{\bf x}||G_f({\bf x})|^2d{\bf x}+|b_1b_2|\int_{\mathbb R^2}\ln|{\bf bw}|\left|\mathcal F^{\mathbb H}[G_f]({\bf bw})\right|^2d{\bf w}\ge D\int_{\mathbb R^2}|G_f({\bf x})|^2d{\bf x}.
\end{eqnarray}
By the equations present in Theorem \ref{relation}, (\ref{log3}) yields
 \begin{eqnarray}
\nonumber &&\int_{\mathbb R^2}\ln|{\bf x}|\left|\sqrt{b_1i}\tilde f({\bf x})\sqrt{b_2j}\right|^2d{\bf x}\\
\nonumber&&+|b_1b_2|\int_{\mathbb R^2}\ln|{\bf bw}|\left|e^{i(c_1w_1^2+e_1w_1)}\mathbb Q^{\mathbb H}_{\mu_1,\mu_2}[f]({\bf w})e^{j(c_2w_2^2+e_2w_2)}\right|^2d{\bf w}\\
\label{log4}&&\ge D\int_{\mathbb R^2}\left|\sqrt{b_1i}\tilde f({\bf x})\sqrt{b_2j}\right|^2d{\bf x}
\end{eqnarray}
Further simplifying (\ref{log4}), we obtain
\begin{eqnarray*}
\nonumber&&|b_1b_2|\int_{\mathbb R^2}\ln|{\bf x}|\left| f({\bf x})\right|^2d{\bf x}
+|b_1b_2|\int_{\mathbb R^2}\ln|{\bf bw}|\left|\mathbb Q^{\mathbb H}_{\mu_1,\mu_2}[f]({\bf w})\right|^2d{\bf w}\\
&&\ge D|b_1b_2|\int_{\mathbb R^2}\left| f({\bf x})\right|^2d{\bf x}.
\end{eqnarray*}
Which implies
\begin{eqnarray}
\nonumber&&\int_{\mathbb R^2}\ln|{\bf x}|\left| f({\bf x})\right|^2d{\bf x}+\int_{\mathbb R^2}\ln|{\bf b}|\left|\mathbb Q^{\mathbb H}_{\mu_1,\mu_2}[f]({\bf w})\right|^2d{\bf w}\\
\label{log5}&&+\int_{\mathbb R^2}\ln|{\bf w}|\left|\mathbb Q^{\mathbb H}_{\mu_1,\mu_2}[f]({\bf w})\right|^2d{\bf w}
\ge D\int_{\mathbb R^2}\left| f({\bf x})\right|^2d{\bf x}.
\end{eqnarray}
By applying Parseval’s identity (\ref{30}) to (\ref{log5}), we obtain
\begin{eqnarray*}
\nonumber&&\int_{\mathbb R^2}\ln|{\bf x}|\left| f({\bf x})\right|^2d{\bf x}+\int_{\mathbb R^2}\ln|{\bf w}|\left|\mathbb Q^{\mathbb H}_{\mu_1,\mu_2}[f]({\bf w})\right|^2d{\bf w}\\
&&\ge (D-\ln|{\bf b}|)\int_{\mathbb R^2}\left| f({\bf x})\right|^2d{\bf x}.
\end{eqnarray*}
Which completes the proof.\\
\end{proof}

In continuation, we shall derive the Hardy’s uncertainty principle for the quaternion quadratic-phase
Fourier transform (\ref{eqn q-qpft}). We  first
recall Hardy’s uncertainty principle for the  QFT.
\begin{lemma}[Hardy’s UP for the two-sided QFT \cite{gen3} ]\label{hardy qft} Let $\alpha$ and $\beta$ be positive constants. For $f(t)\in L^2(\mathbb R^2,\mathbb H),$ if \\
$f({\bf x})\le ce^{-\alpha |{\bf x}|^2}$ and $|\mathcal F^{\mathbb H}[f]({\bf w})|\le c'e^{-\beta|{\bf w}|^2}$,\quad ${\bf u,w}\in \mathbb R^2$\\
with some positive constants c,c'.Then, there are the following three cases to occur:\\
(1)if $\alpha\beta>\frac{1}{4},$ then $f({\bf x})\equiv 0;$\\
(2)if $\alpha\beta=\frac{1}{4},$ then $f({\bf x})=ke^{-\alpha|{\bf x}|},$ for any constant $k$\\
(3)if $\alpha\beta<\frac{1}{4},$ then there are many infinite such functions $f({\bf x})$.
\end{lemma}
Motivated and inspired by Hardy's UP for the two-sided QFT, we establish Hardy's UP for the  Q-QPFT.

\begin{theorem}[Hardy’s UP for the  Q-QPFT]
 Let $\alpha$ and $\beta$ be positive constants. For $f(t)\in L^2(\mathbb R^2,\mathbb H),$ if \\
$f({\bf x})\le Ce^{-\alpha |{\bf x}|^2}$ and $\left|\mathbb Q^{\mathbb H}_{\mu_1,\mu_2}[f]({\bf\frac{ w}{b}})\right|\le C'e^{-\beta|{\bf w}|^2}$,\quad ${\bf u,w}\in \mathbb R^2$\\
with some positive constants C,C'.Then, there are the following three cases to occur:\\
(1)if $\alpha\beta>\frac{1}{4},$ then $f({\bf x})\equiv 0;$\\
(2)if $\alpha\beta=\frac{1}{4},$ then $f({\bf x})=e^{i(a_1x_1^2+d_1x_1)}Ke^{-\alpha|{\bf x}|}e^{j(a_2x_2^2+d_2x_2)},$ for any constant $K;$\\
(3)if $\alpha\beta<\frac{1}{4},$ then there are many infinite such functions $f({\bf x})$.
\end{theorem}
\begin{proof}
Assuming $f=G_f$ defined in Theorem \ref{relation}, it follows that
\begin{equation}\label{h1}
|G_f({\bf x})|\le ce^{-\alpha |{\bf x}|^2} \qquad {\bf x}\in \mathbb R^2
\end{equation}
and
\begin{equation}\label{h2}
|\mathcal F^{\mathbb H}[G_f]({\bf w})|\le c'e^{-\beta |{\bf w}|^2} \qquad {\bf w}\in \mathbb R^2.
\end{equation}
Thus by Lemma \ref{hardy qft},there are the following three cases to occur:\\
(1)if $\alpha\beta>\frac{1}{4},$ then $G_f({\bf x})\equiv 0;$\\
(2)if $\alpha\beta=\frac{1}{4},$ then $G_f({\bf x})=ke^{-\alpha|{\bf x}|},$ for any real constant $k$\\
(3)if $\alpha\beta<\frac{1}{4},$ then there are many infinite such functions $G_f({\bf x})$.\\
Now it is clear from Theorem \ref{relation} and equations (\ref{h1}),(\ref{h2}) that
\begin{equation}\label{h3}
|G_f({\bf x})|=\sqrt{b_1}|f({\bf x})|\sqrt{b_2}\le ce^{-\alpha |{\bf x}|^2} \qquad {\bf x}\in \mathbb R^2
\end{equation}
and
\begin{equation}\label{h4}
|\mathcal F^{\mathbb H}[G_f]({\bf w})|=\left|\mathbb Q^{\mathbb H}_{\mu_1,\mu_2}[f]({\bf \frac{w}{b}})\right|\le c'e^{-\beta |{\bf w}|^2}\qquad {\bf w}\in \mathbb R^2.
\end{equation}
From (\ref{h3})and (\ref{h4}), we have\\
$|f({\bf x})|\le Ce^{-\alpha |{\bf x}|^2}$ \quad ${\bf u}\in \mathbb R^2 $ and $\left|\mathbb Q^{\mathbb H}_{\mu_1,\mu_2}[f]({\bf \frac{w}{b}})\right|\le C'e^{-\beta|{\bf w}|^2}$,\quad ${\bf w}\in \mathbb R^2$\\
where $C=\frac{c}{\sqrt{b_1b_2}}$ and $C=c'.$ Thus we have following conclusions:\\
(1)if $\alpha\beta>\frac{1}{4},$ then $f({\bf x})\equiv 0$  for $G_f({\bf x})\equiv 0 ;$\\
(2)if $\alpha\beta=\frac{1}{4},$ it yields $f({\bf x})=e^{i(a_1x_1^2+d_1x_1)}Ke^{-\alpha|{\bf x}|}e^{j(a_2x_2^2+d_2x_2)},$ where  $K=\frac{1}{\sqrt{b_1i}}k\frac{1}{\sqrt{b_2j}}$ owing to Theorem \ref{relation}.\\
(3)if $\alpha\beta<\frac{1}{4},$ then it is clear there are many infinite such functions $f({\bf x})$.
\end{proof}
Which completes the proof.\\

Now, using the relationship between the proposed transform (Q-QPFT) and QFT, we obtain Beurling’s uncertainty principle for the Q-QPFT.
First we recall the Beurling’s uncertainty principle for the QFT.
\begin{lemma}[Beurling’s UP for the two-sided QFT \cite{gen4}]\label{ber qft}Let $f({\bf x})\in L^2(\mathbb R^2,\mathbb H)$ and $d\ge 0$ such that
\begin{equation*}
\int_{\mathbb R^2}\int_{\mathbb R^2}\dfrac{|f({\bf x})|\left|\mathcal F^{\mathbb H}[f]({\bf w})\right|}{(1+|x|+|w|)^d}e^{|{\bf x}||{\bf w}|}d{\b x}d{\bf w}<\infty,
\end{equation*}
then $f({\bf x})=P({\bf x})e^{-k|{\bf x}|^2},$ where $k>0$ and $P$ is a polynomial of degree $< \frac{d-2}{2}.$ In particular,$f=0$ when $d\le 2.$
\end{lemma}
By applying Theorem \ref{relation} and Lemma (\ref{ber qft}) , we extend the validity of Beurling’s UP for the Q-QPFT.

\begin{theorem}[Beurling’s UP for the Q-QPFT ]\label{ber q-qpft}Let $f({\bf x})\in L^2(\mathbb R^2,\mathbb H)$ and $d\ge 0$ satisfying
\begin{equation*}
\int_{\mathbb R^2}\int_{\mathbb R^2}\dfrac{|f({\bf x})|\left|\mathbb Q^{\mathbb H}_{\mu_1,\mu_2}[f]({\bf w})\right|}{(1+|{\bf x}|+|{\bf bw}|)^d}e^{|{\bf x}||{\bf b w}|}d{\bf x}d{\bf w}<\infty,
\end{equation*}
then $f({\bf x})=e^{i(a_1x_1^2+d_1x_1)}P'({\bf x})e^{-k|{\bf x}|^2}e^{j(a_2x_2^2+d_2x_2)},$ where $k>0$ and $P'({\bf x})=\frac{1}{\sqrt{b_1i}}P({\bf x})\frac{1}{\sqrt{b_2j}}$ is a polynomial of degree $< \frac{d-2}{2}.$ In particular,$f=0$ when $d\le 2.$
\end{theorem}

\begin{proof}
If we take $f=G_f$ as defined in Theorem \ref{relation}, then it follows that
\begin{eqnarray*}
&&\int_{\mathbb R^2}\int_{\mathbb R^2}\dfrac{|G_f({\bf x})|\left|\mathcal F^{\mathbb H}[G_f]({\bf w})\right|}{(1+|x|+|w|)^d}e^{|{\bf x}||{\bf w}|}d{\bf x}d{\bf w}\\
&&=
\int_{\mathbb R^2}\int_{\mathbb R^2}\sqrt{b_1b_2}\dfrac{|f({\bf x})|\left|\mathbb Q^{\mathbb H}_{\mu_1,\mu_2}[f]({\bf \frac{w}{b}})\right|}{(1+|x|+|w|)^d}e^{|{\bf x}||{\bf w}|}d{\bf x}d{\bf w}<\infty.
\end{eqnarray*}
Hence by Lemma \ref{ber qft}, we must have $G_f({\bf x})=P({\bf x})e^{-k|{\bf x}|^2}.$
Now,
\begin{eqnarray*}
&&\int_{\mathbb R^2}\int_{\mathbb R^2}\dfrac{|G_f({\bf x})|\left|\mathcal F^{\mathbb H}[G_f]({\bf w})\right|}{(1+|{\bf x}|+|{\bf w}|)^d}e^{|{\bf x}||{\bf w}|}d{\b x}d{\bf w}\\
&&=
\int_{\mathbb R^2}\int_{\mathbb R^2}\sqrt{b_1b_2}\dfrac{|f({\bf x})|\left|\mathbb Q^{\mathbb H}_{\mu_1,\mu_2}[f]({\bf \frac{w}{b}})\right|}{(1+|{\bf x}|+|{\bf w}|)^d}e^{|{\bf x}||{\bf w}|}d{\bf x}d{\bf w}\\
&&=(b_1b_2)^{\frac{3}{2}}\int_{\mathbb R^2}\int_{\mathbb R^2}\dfrac{|f({\bf x})|\left|\mathbb Q^{\mathbb H}_{\mu_1,\mu_2}[f]({\bf w})\right|}{(1+|{\bf x}|+|{\bf bw}|)^d}e^{|{\bf x}||{\bf bw}|}d{\bf x}d{\bf w}\\
&&<\infty.
\end{eqnarray*}
As $b_1,b_2$ are finite real numbers, therefore we can write
\begin{equation*}\int_{\mathbb R^2}\int_{\mathbb R^2}\dfrac{|f({\bf x})|\left|\mathbb Q^{\mathbb H}_{\mu_1,\mu_2}[f]({\bf w})\right|}{(1+|{\bf x}|+|{\bf bw}|)^d}e^{|{\bf x}||{\bf bw}|}d{\bf x}d{\bf w}<\infty.
\end{equation*}
Since $G_f({\bf x})=\sqrt{b_1i}\tilde f({\bf x})\sqrt{b_2j}=\sqrt{b_1i}e^{-i(a_1x_1^2+d_1x_1)} f({\bf x})e^{-j(a_2x_2^2+d_2x_2)}\sqrt{b_2j},$
which implies\\
$f({\bf x})=e^{i(a_1x_1^2+d_1x_1)}P'({\bf x})e^{-k|{\bf x}|^2}e^{j(a_2x_2^2+d_2x_2)}.$  In particular,$f=0$ on account $G_f({\bf x})=0$ when $d\le 2.$
\end{proof}
Which completes the proof.\\

Towards the end of this section, we establish Donoho-Stark’s uncertainty principle for the Q-QPFT by considering relationship between the proposed transform (Q-QPFT) and QFT. Let us begin with the definition.

\begin{definition}\cite{2s26}A quaternion function $f\in L^2(\mathbb R^2,\mathbb H)$ is said to be $\varepsilon-$concentrated on a measurable set $E\subseteq\mathbb R^2,$ if
\begin{equation*}
\left(\int_{\mathbb R^2\setminus E}|f({\bf x})|^2d{\bf x}\right)^{1/2}\le\varepsilon\|f\|_2.
\end{equation*}
\end{definition}
\begin{lemma}[Donoho-Stark’s UP for the two-sided QFT \cite{2s9,2s26}]\label{lem dh}Let $f\in L^2(\mathbb R^2,\mathbb H)$ with $f\ne 0$ be $\varepsilon_{E_1}-$concentrated on  $E_1\subseteq\mathbb R^2$ and $\mathcal F^{\mathbb H}[f]$ be $\varepsilon_{E_2}-$concentrated on  $E_2\subseteq\mathbb R^2.$ Then
\begin{equation*}
|E_1||E_2|\ge2\pi(1-\varepsilon_{E_1}-\varepsilon_{E_2})^2.
\end{equation*}
\end{lemma}

\begin{theorem}[Donoho-Stark’s UP for the Q-QPFT] Assuming that non-zero signal   $f$ in $L^2(\mathbb R^2,\mathbb H)$ is a $\varepsilon_{E_1}-$concentrated on  $E_1\subseteq\mathbb R^2$ and $\mathbb Q^{\mathbb H}_{\mu_1,\mu_2}[f]({\bf w})$ is $\varepsilon_{E_2}-$concentrated on  $E_2\subseteq\mathbb R^2.$ Then
\begin{equation*}
|E_1||E_2|\ge\frac{2\pi}{{\bf b}}(1-\varepsilon_{E_1}-\varepsilon_{E_2})^2.
\end{equation*}
\end{theorem}
\begin{proof}
By Theorem \ref{relation}, we have
\begin{eqnarray}\label{dh1}
\left|\mathbb Q^{\mathbb H}_{\mu_1,\mu_2}[f]({\bf \frac{w}{b}})\right|=\left|\mathcal F^{\mathbb H}[G_f]({\bf w})\right|
\end{eqnarray}
Since $\mathbb Q^{\mathbb H}_{\mu_1,\mu_2}[f]({\bf w})$ is $\varepsilon_{E_2}-$concentrated on  $E_2\subseteq\mathbb R^2,$ therefore  (\ref{dh1}) implies  $\mathcal F^{\mathbb H}[G_f]$ is $\varepsilon_{E_2}-$concentrated on  ${\bf b}E_2\subseteq\mathbb R^2.$\\
Also from (\ref{func hf}), we have
\begin{equation}\label{dh2}
|G_f({\bf x})|=\sqrt{b_1b_2}|f({\bf x})|.
\end{equation}
By the given condition $f$ is $\varepsilon_{E_1}-$concentrated on  $E_1\subseteq\mathbb R^2,$i.e.

\begin{equation}\label{dh3}
\left(\int_{\mathbb R^2\setminus E_1}|f({\bf x})|^2d{\bf x}\right)^{1/2}\le\varepsilon\|f\|_2.
\end{equation}
From (\ref{dh2}) and (\ref{dh3}), we obtain
\begin{equation*}
\left(\int_{\mathbb R^2\setminus E_1}|G_f({\bf x})|^2d{\bf x}\right)^{1/2}\le\varepsilon\|G_f\|_2.
\end{equation*}
Which implies $G_f\in L^2(\mathbb R^2,\mathbb H)$ is $\varepsilon_{E_1}-$concentrated on  $E_1\subseteq\mathbb R^2.$
Hence we proved that the function $G_f({\bf x})$ and its QFT $\mathcal F^{\mathbb H}[G_f]({\bf w})$ are $\varepsilon_{E_1}-$concentrated on  $E_1\subseteq\mathbb R^2$ and $\varepsilon_{E_2}-$concentrated on  ${\bf b}E_2\subseteq\mathbb R^2,$ respectively. Therefore by Lemma \ref{lem dh}, we have
\begin{equation*}
|E_1||{\bf b}E_2|\ge2\pi(1-\varepsilon_{E_1}-\varepsilon_{E_2})^2,
\end{equation*}
so that
\begin{equation*}
|E_1||E_2|\ge\frac{2\pi}{|{\bf b}|}(1-\varepsilon_{E_1}-\varepsilon_{E_2})^2,
\end{equation*}
Which completes the proof.
\end{proof}

\section{Conclusion}\label{sec 4}

In the study, we have accomplished three major objectives: first, we have introduced the notion of quaternion
quadratic-phase Fourier transform (Q-QPFT). Second, we establish the fundamental properties of the proposed transform,
including the parseval's formula, inversion formula, shift and modulation by using the fundamental relationship between Q-QPFT and QFT.
 Third, we investigate some different forms of UPs associated with Q-QPFT including   Heisenberg UP, logarithmic UPs, Hardy's UP, Beurling’s UP, and Donoho-Stark’s UP. In our future works we shall study the short-time quadratic-phase Fourier transform in the quaternion setting.
\section*{Declarations}
\begin{itemize}
\item  Availability of data and materials: The data is provided on the request to the authors.
\item Competing interests: The authors have no competing interests.
\item Funding: No funding was received for this work
\item Author's contribution: Both the authors equally contributed towards this work.
\item Acknowledgements: This work is supported by the  Research  Grant\\
(No. JKST\&IC/SRE/J/357-60) provided by JKST\&IC,  UT of J\& K,
India.

\end{itemize}

{\bf{References}}

\end{document}